\documentclass[a4paper,12pt]{amsart}
\usepackage[french, english]{babel}
\usepackage{amsfonts}
 \usepackage{color}
\usepackage{pifont}
\usepackage{latexsym,amsmath,amssymb,cite,amsthm}
\usepackage[colorlinks=true, linkcolor=blue,  anchorcolor=blue,citecolor=blue]{hyperref}
\usepackage{hyperref}
\usepackage[pagewise]{lineno}

\usepackage{color,eucal,enumerate,mathrsfs}
\usepackage[normalem]{ulem}
\usepackage{amsmath,amssymb,epsfig,bbm}

\numberwithin{equation}{section}
\theoremstyle{plain}
\newtheorem{theorem}{Theorem}[section]

\newtheorem{lemma}[theorem]{Lemma}

\newtheorem{example}[theorem]{Example}
\theoremstyle{definition}

\newtheorem{definition}[theorem]{Definition}
\theoremstyle{remark}
\newtheorem{assumption}[theorem]{\textbf{Assumption}}
\newtheorem{remark}[theorem]{Remark}

\newcommand{\lmt}[2]{\mathop{\lim}_{{#1} \rightarrow {#2}} }

\newcommand{\lmts}[2]{\mathop{\overline{\lim}}_{{#1} \rightarrow {#2}} }
\newcommand{\lmti}[2]{\mathop{\underline{\lim}}_{{#1} \rightarrow {#2}} }

\newcommand{\Ric}{{\rm{Ric}}}

\renewcommand{\H}{{\mathrm{Hess}}}

\newcommand{\mm}{\mathfrak m}
\newcommand{\ms}{(X,\d,\mm)}

\newcommand{\rcdkn}{{\rm RCD}(K, N)}
\newcommand{\rcd}{{\rm RCD}(K, \infty)}

\newcommand{\E}{\mathcal{E}}

\newcommand{\N}{\mathbb{N}}

\newcommand{\R}{\mathbb{R}}

\newcommand{\Lip}{\mathop{\rm Lip}\nolimits}

\renewcommand{\d}{{\mathrm d}}

\newcommand{\restr}[1]{\lower3pt\hbox{$|_{#1}$}}

\newcommand{\nchi}{{\raise.3ex\hbox{$\chi$}}}

\begin{document}

\title[Maz'ya-Shaposhnikova's theorem and rigidity]{\footnotesize  On the asymptotic behaviour of the  fractional Sobolev seminorms in  metric measure spaces:  asymptotic volume ratio,  volume entropy  and rigidity}

\author[B. ~Han]{Bang-Xian Han}
\address{School of   Mathematical Sciences, University of Science and Technology of China (USTC), 230026, Hefei, China}
\email{hanbangxian@ustc.edu.cn}

\author[A.~Pinamonti]{Andrea Pinamonti}
\address{Dipartimento di Matematica\\ Universit\'a di Trento\\
Via Sommarive, 14, 38123 Povo TN}
\email{andrea.pinamonti@unitn.it}

\date{\today}
\bibliographystyle{abbrv}

\begin{abstract}
We study the asymptotic behaviour of suitably defined seminorms in general metric measure spaces. As  particular cases we provide new and shorter proofs of the  Maz'ya-Shaposhnikova's  theorem \cite{MS} 
with general mollifiers as well as of  Ludwig's result \cite{Ludwig14} concerning anisotropic Sobolev seminorms. 
Our result  provides new spaces satisfying an asymptotic formula involving the asymptotic volume ratio (and the volume entropy),   non-trivial examples fit our setting  includes Carnot groups and metric measure spaces with synthetic lower Ricci curvature bound.   Moreover, we prove two  rigidity theorems which are new even in the smooth setting.
\end{abstract}

\maketitle

\section{Introduction}\label{first section}

About 20 years ago,   Bourgain, Brezis and Mironescu \cite{BBM} on one hand and  Maz’ya and Shaposhnikova \cite{MS}  on the other, revealed that fractional $(s,p)$-seminorms  can be used to recover the   $L^p$-norm and the $W^{1, p}$-seminorms when $s\to 0^+$ and $s\to 1^{-}$ respectively.  More precisely,  for any $s\in (0, 1)$, $N\in \N$ and $p\geq 1$, the fractional Sobolev space $W^{s, p}(\R^N)$ is defined as the set of $L^p(\R^N)$ functions with finite seminorm
\[
\|f\|_{W^{s, p}}:=\left (\int_{\R^N} \int_{\R^N} \frac{|f(x)-f(y)|^p}{|x-y|^{N+sp}} \,\d\mathcal L^N (x) \d \mathcal L^N (y) \right)^{\frac 1p}
\]
where $\mathcal L^N$ denotes the Lebesgue measure on $\R^N$. The following  important asymptotic formulas have been proved in \cite{BBM} and \cite{MS} respectively: for any $f\in \cup_{0<s<1} W^{s, p}(\R^N) \cap W^{1, p}(\R^N)$,   it holds
 \begin{equation}\label{BBM:intro}
 \mathop{\lim}_{s \uparrow 1}~ (1-s)\|f\|^p_{W^{s, p}}= {K }  \| \nabla f\|_{L^p(\R^N)}^p  \tag{BBM}
\end{equation}
and
 \begin{equation}\label{MS:intro}
 \mathop{\lim}_{s \downarrow 0} ~s\|f\|^p_{W^{s, p}}= {L }  \|  f\|_{L^p(\R^N)}^p  \tag{MS}
\end{equation}
where $K={K_{p, N} }$ and $L={L_{p, N} }$ are positive constants depending only on $p$ and $N$.
 
Both formulas \eqref{BBM:intro} and \eqref{MS:intro} have been widely studied, cited in hundreds of papers  in the fields of analysis,  probability theory and  geometry,  and generalized to many different situations, see for example \cite{Cap,Cianchi21,HanPinamonti-BBM,Gorny20,Kre,Lam,Ludwig14, MMilman05TAMS, MMilman05JFA, Mal,Ngu,Pin1,Pin2,Ponce} and the references therein. Therefore, we believe it is an interesting problem to find more  non-trivial examples satisfying such asymptotic formulas, and learn more about the structural constants $K$ and $L$ appearing in the formulas.

In the present paper we study \eqref{MS:intro} in the setting of general metric measure space and general mollifiers.
Following a completely different proof than the one proposed in \cite{MS} we prove that \eqref{MS:intro} holds in great generality and that the mollifier $\frac{s}{|x-y|^{n+sp}}$ can be replaced by more general ones (Theorem \ref{th:MS2}). As a consequence we provide a new proof of \cite{MS},  as well as of \cite{Ludwig14} without using Blaschke-Petkantschin formula.  

We also prove a surprising link between the Maz'ya-Shaposhnikova's formula  and  the asymptotic volume ratio of the metric measure space (Lemma \ref{lemma1} and Theorem \ref{th:MS}). Using this very precise and geometric information we provide several interesting examples of spaces satisfying a Maz'ya-Shaposhnikova's type formula.  In addition, we find a new dimension-free asymptotic formula involving volume entropy (Theorem \ref{th:MS1}).
Finally, applying our result to the setting of metric measure spaces with synthetic curvature-dimension condition \`a la Lott-Sturm-Villani,  we  obtain two  sharp estimates on the corresponding structural constants, and prove two rigidity theorems (Theorem \ref{eq1:rigidity1} and Theorem \ref{rigidity2}). To our knowledge,   these results  are new even in the smooth setting.


\vspace{0.5cm} 

\noindent
{\bf Plan  of the paper.} 

In subsection \ref{sect: mainth} we prove our main result, i.e. Theorem \ref{th:MS2} .
In the next subsection \ref{sect:appl}, we apply our main theorem with two specific mollifiers and we prove Lemma \ref{lemma1}, moreover we provide several examples of metric measure spaces where Theorem \ref{th:MS2} and  Lemma \ref{lemma1} hold. 
Finally, in subsection \ref{splitting} we prove two rigidity results Theorem \ref{eq1:rigidity1} and Theorem \ref{rigidity2}.

\section{Main results}
\subsection{General theory}\label{sect: mainth}
In this paper,  a metric measure space  is a triple $\ms$ where $(X, \d)$ is a complete and separable metric space,  $\mm$ is a locally finite non-negative Borel measure on $X$ with full support.

Let $(\rho_n)_{n \in \N}$ be a sequence of mollifiers i.e. each \[\rho_n:(X\times X)\setminus \{(x,y)\in X\times X\ |\ x=y\}\to (0,\infty)\] is measurable and it satisfies the following set of assumptions:
\begin{assumption}\label{assumption2}
\item [1)] There is a sequence of functions $(\tilde\rho_n)_{n\in\N}$ with $\tilde \rho_n:(0,\infty)\to (0,\infty)$ continuous such that
\[
\rho_n (x, y)=\tilde\rho_n \big( \d(x, y)\big)\ \mbox{for all}\ x,y\in X,\ x\neq y.
\]
Moreover, there exists $n_0\in \N$ such that $\tilde\rho_n(r)$ is non-increasing in $r$ for any $n>n_0$,  there exists $r_1>0$ such that $(\tilde\rho_n)_{n\in\N}$ is non-increasing  in $n\in \N$ for any $r\in (r_1,\infty)$, and
\begin{equation}\label{limit}
\lim_{n\to\infty} \tilde\rho_n(r)=0,\quad \forall r\in (r_1,\infty).
\end{equation}

\item [2)] There exists $r_0>0$ such that for any $n,m\in\mathbb{N}$ with $n >m$ the function
\[
(r_0, +\infty) \ni  r \to \frac{\tilde \rho_n(r)}{\tilde \rho_m(r)} 
\]
is non-decreasing.
\item [3)]  There is a structural constant $L\in [0,\infty)$ such that the following limit holds
 \begin{equation}\label{eq1:assumption}
\lmt{\delta}{\infty} \lmt{n}{\infty}\int_{B^c_\delta(x) }  \rho_n(x, y) \,\d \mm(y)=L,~~~\forall x\in X
 \end{equation}
 where $B^c_\delta(x) $ denotes the complement of the open ball centred at $x$ with radius $\delta>0$. 
\end{assumption}
 
 \bigskip

Let $p>1$, $n\in \N$ and $(\rho_n)_{n\in\N}$ be a family of mollifiers. We define the $W(n, p)$ as the set of functions in $L^p(X, \mm)$ with finite (semi)norm $\E_n(\cdot)$,  defined by
\begin{equation*}
\E_n(u):= \left (\int_X \int_X {|u(x)-u(y)|^p} \rho_n(x, y)\,\d\mm(x) \d\mm(y)\right)^{\frac 1 p}.
\end{equation*}
We are now ready to prove a generalized  Maz'ya-Shaposhnikova's asymptotic formula under the Assumption \ref{assumption2}.
\begin{theorem}[Generalized  Maz'ya-Shaposhnikova's formula]\label{th:MS2}
Given $p>1$ and  a metric measure space $\ms$.  Let $(\rho_n)_{n\in \N}$ be a family of mollifiers.
For any $u\in  \cup_{n\in \N} W(n, p)$,  there exists the limit
\begin{equation}\label{eq1:MS}
\mathop{\lim}_{n\to \infty} \E_n^p(u)=2L\|u\|^p_{L^p}
\end{equation}
where $L$ is the constant given  in \eqref{eq1:assumption}.
\end{theorem}

\begin{proof}

 Fix $x_0 \in X$, for any $R \gg r_0$ ($r_0>0$ is as in Assumption \ref{assumption2}-2)) ,  we have  the following decomposition of $X\times X$:
\begin{equation*}
\left\{
             \begin{array}{l}
             {A:=\big \{(x, y): \d(x, y)\leq R \big \}} \\
             {B:= \big \{(x, y): \d(x, y)> R \big \} \cap \big \{(x, y): \d(y, x_0)>2\d(x, x_0) \big \}} \\
            {B':= \big \{(x, y): \d(x, y)>R \big \} \cap \big \{(x, y): \d(y, x_0)< \frac 12 \d(x, x_0) \big \}} \\
               {C:=   \big \{(x, y): \d(x, y)> R \big \} \cap \big \{(x, y): \d(x, x_0) \leq \d(y, x_0)\leq 2\d(x, x_0) \big \}} \\
                  {C':= \big \{(x, y): \d(x, y)> R \big \} \cap \big \{(x, y):\frac 12  \d(x, x_0) \leq \d(y, x_0)<\d(x, x_0) \big \}} \\
             \end{array}  
        \right.
\end{equation*}

According to this  decomposition,  and by symmetry (w.r.t. $x, y$), we divide  $\E_n^p(u)$  into the following three parts 
\begin{eqnarray*}
\E^p_n(u) &=& \underbrace{  \int_{A} {|u(x)-u(y)|^p}{\rho_n(x, y)} \,\d \mm( x)\,\d \mm(y)}_{I:=I(R, n)}\\
&& + 2\underbrace{  \int_{{B} } {|u(x)-u(y)|^p}{\rho_n(x, y)} \,\d \mm(x)\,\d \mm(y)}_{II:=II(R, n)}\\
&& + \underbrace{ \int_{{C}\cup C' } {|u(x)-u(y)|^p}{\rho_n(x, y)} \,\d \mm(x)\,\d \mm(y)}_{III:=III(R, n)}.
\end{eqnarray*}

By hypothesis, there is $n_0 \in \N$ such that $u \in W(n_0, p)$. Concerning the term $I$, for any $n>n_0$, by  Assumption \ref{assumption2}, it holds
\begin{eqnarray*}
I&=&  \int_{X} \left( \int_{B_R(y)} {|u(x)-u(y)|^p}{\rho_{n_0}(x, y)} \frac{\rho_{n}(x, y)}{\rho_{n_0}(x, y)}\,\d \mm( x)\right )\,\d \mm(y)\\
&\leq &   \int_{X} \left( \int_{B_R(y)} {|u(x)-u(y)|^p}{\rho_{n_0}(x, y)} \frac{\tilde \rho_{n}(R)}{\tilde\rho_{n_0}(R)}\,\d \mm( x) \right )\,\d \mm(y)\\
&\leq&  \E^p_{n_0}(u)  \frac{\tilde \rho_{n}(R)}{\tilde\rho_{n_0}(R)}.
\end{eqnarray*}
By Assumption \ref{assumption2}-1) we get
\begin{equation}\label{eq1.1:MS}
\lmt{n}{+\infty} I(R, n) =0.
\end{equation}

On the other hand,  for any  $R>0$, we have
\begin{eqnarray*}
&&B_1(R, n)\\&:=&\int_{ B} {|u(y)|^p}{\rho_n(x, y)} \,\d \mm(x) \,\d \mm(y)\\
&=&\int_{X} |u(y)|^p \left( \int_{ \big \{x ~ |~ \d(x, y)>R,~ \d(y, x_0)>2\d(x, x_0) \big \}} {\rho_n(x, y)} \,\d \mm(x) \right) \,\d \mm(y)\\
&\leq &\int_{X} {|u(y)|^p} \left( \int_{ \big  \{x~ |~  \frac32 \d(y, x_0)  \geq \d(x, y)>\frac12 \d(y, x_0) \vee R\big \}}  {\rho_n(x, y)} \,\d \mm(x) \right ) \,\d \mm(y)\\
&=&    \int_{X} {|u(y)|^p} \Big(\int_{ \big  \{x~ |~  \d(x, y)>\frac12 \d(y, x_0)\vee R \big \}}  {\rho_n(x, y)} \,\d \mm(x)\\
&&-\int_{ \big  \{x~ |~  \d(x, y)>\frac32 \d(y, x_0)\vee R \big \}}  {\rho_n(x, y)} \,\d \mm(x)   \Big)  \,\d \mm(y).
\end{eqnarray*}

By Assumption \ref{assumption2}-1) and 3),  and the monotone convergence theorem, we have
\begin{eqnarray*}
&&\lmt{R}{\infty} \lmt{n}{\infty}  \int_{X} {|u(y)|^p} \Big(\int_{ \big  \{x~ |~  \d(x, y)>\frac12 \d(y, x_0)\vee R \big \}}  {\rho_n(x, y)} \,\d \mm(x) \Big)  \,\d \mm(y)\\
&=&\lmt{R}{\infty} \lmt{n}{\infty}  \int_{X} {|u(y)|^p} \Big(\int_{ \big  \{x~ |~  \d(x, y)>\frac32 \d(y, x_0)\vee R \big \}}  {\rho_n(x, y)} \,\d \mm(x) \Big)  \,\d \mm(y)\\
&=&  L\| u \|^p _{L^p}.
\end{eqnarray*}
Hence 
\begin{equation}\label{eq1.11:MS}
\lmt{R}{\infty} \lmt{n}{\infty} B_1(R, n)=0.
\end{equation}
Notice that for any $x,y,x_0\in X$, by triangle inequality
\[
B^c_{4\d(x, x_0)}(x) \subset \big \{y\in X: \d(y, x_0)>2\d(x, x_0) \big \}\subset B^c_{\d(x, x_0)}(x).
\]
So we also have
\begin{eqnarray*}
B_2(R, n)&:=& \int_{B} {|u(x)|^p}{\rho_n(x, y)} \,\d \mm(x) \,\d \mm(y)\\
&=&\int_{X} |u(x)|^p \left( \int_{ \big \{y ~ |~\d(y, x)>R, ~ \d(y, x_0)>2\d(x, x_0) \big \}} {\rho_n(x, y)} \,\d \mm(y) \right) \,\d \mm(x)\\
&\geq& \int_{X} |u(x)|^p \left( \int_{ B^c_{4\d(x, x_0)\vee R}(x)} {\rho_n(x, y)} \,\d \mm(y) \right) \,\d \mm(x)
\end{eqnarray*}
and
\[
 B_2(R,n) \leq \int_{X} |u(x)|^p \left( \int_{ B^c_{\d(x, x_0)\vee R}(x)} {\rho_n(x, y)} \,\d \mm(y) \right) \,\d \mm(x).
\]
By Assumption \ref{assumption2}-1) and 3),  and monotone convergence theorem, we obtain
\begin{equation}\label{eq2:MS}
\lmt{R}{\infty} \lmt{n}{\infty} B_2(R, n) = L\|u\|_{L^p}^p.
\end{equation}
By triangle inequality we have
\begin{eqnarray*}
&&\left (B_2\right)^\frac1p-\left (B_1 \right)^\frac1p\\
&\leq&II^\frac1p= \left (  \int_{{B} } {|u(x)-u(y)|^p}{\rho_n(x, y)} \,\d \mm(x)\,\d \mm(y) \right)^\frac1p\\
&\leq& \left (B_2 \right)^\frac1p+\left (B_1 \right)^\frac1p.
\end{eqnarray*}
Combining with  \eqref{eq1.11:MS} and  \eqref{eq2:MS}, we get
\begin{equation}\label{eq4:MS}
\lmt{R}{\infty} \lmt{n}{\infty} II (R, n)= L\|u\|_{L^p}^p.
\end{equation}
We observe that 
\[
C=   \left\{(x, y)\in X\times X~|~ \d(x, y)> R \big \} \cap \big \{(x, y)\in X\times X~|~ \d(x, x_0) \leq \d(y, x_0)\leq 2\d(x, x_0)\right\}.
\]
By triangle inequality
\begin{equation}	\label{10}
C\subset \left\{(x,y)\in X\times X~|~\d(x,y)>R, ~\d(x,x_0)>\frac{R}{3}\right\}
\end{equation}
and 
\begin{equation}\label{2}
C\subset \left\{(x,y)\in X\times X~|~\d(x,y)>R, ~\d(y,x_0)>\frac{R}{2}\right\}.
\end{equation}
Thus by \eqref{10} and \eqref{2} we can estimate 
\begin{eqnarray*}
III&\leq & 2\int_{{C} } {|u(x)-u(y)|^p}{\rho_n(x, y)} \,\d \mm(x)\,\d \mm(y)\\
&\leq & 2^{p-1}\left(\int_C|u(x)|^p\rho_n(x,y)\,\d \mm(x)\,\d\mm(y)+\int_C|u(y)|^p\rho_n(x,y)\,\d \mm(x)\,\d\mm(y)\right)\\
&\leq & 2^{p-1}\int_{\d(x, x_0)>  \frac R3} |u(x)|^p \Big ( \int_{\d(x, y)>  R}  \rho_{n}(x, y)  \,\d \mm(y)\Big ) \,\d \mm(x) \\
&+& 2^{p-1}\int_{\d(y, x_0)>  \frac R2} |u(y)|^p \Big ( \int_{\d(x, y)>  R}  \rho_{n}(x, y)  \,\d \mm(x)\Big ) \,\d \mm(y)\\
&\leq & 2^p\int_{\d(x, x_0)>  \frac R3} |u(x)|^p \Big ( \int_{\d(x, y)>  R}  \rho_{n}(x, y)  \,\d \mm(y)\Big ) \,\d \mm(x).
\end{eqnarray*}
So by Assumption \ref{assumption2} and the monotone convergence theorem
$$ \lmts{n}{\infty} III(R, n) \leq 2^p\int_{\d(x, x_0)>  \frac R3} |u(x)|^p  \lmts{n}{\infty}  \Big ( \int_{\d(x, y)>  R}  \rho_{n}(x, y)  \,\d \mm(y)\Big ) \,\d \mm(x).$$
and
$$\lmt{R}{\infty} \lmts{n}{\infty} III(R, n) =0.$$

The conclusion follows combining \eqref{eq1.1:MS} and  \eqref{eq4:MS} obtained above.

\end{proof}

 \subsection{Applications and Examples}\label{sect:appl}
In this section we will apply our main theorem to some important  spaces  and to some specific mollifiers. Doing so we not only extend Maz'ya-Shaposhnikova's theorem on Euclidean spaces (and Ludwig's theorem on finite dimensional Banach spaces), but we also show that the key point in the asymptotic formula  is  the so-called `asymptotic volume ratio' which describe the growth of the volume of  a geodesic ball at infinity.
 
\begin{definition} \label{def:avr}
A metric measure space $\ms$ admits the  \textbf{asymptotic volume ratio at $x_0\in X$} with some $N\in (0,+\infty)$, denoted by ${\rm AVR}\ms(x_0, N)$,  provided 
\begin{equation*}
{\rm AVR}\ms(x_0, N):=\lmt{r}{+\infty} \frac {\mm\big (B_r(x_0)\big)} {r^N}\in [0, +\infty].
\end{equation*}
\end{definition}
\begin{lemma}\label{lemma:avr}
Let $\ms$ be a metric measure space admitting finite asymptotic volume ratio  ${\rm AVR}\ms(x_0)$ at $x_0\in X$ with some $N >0$. Then   $\ms$ admits the asymptotic volume ratio at any $x_1\in X$ with the same constant $N$ and
\begin{equation}
{\rm AVR}\ms(x_1, N)={\rm AVR}\ms(x_0, N).
\end{equation}

Furthermore, if ${\rm AVR}\ms(x_0, N)\in (0, +\infty)$,  then $N$ is unique.\\
If ${\rm AVR}\ms(x_0, N)=0$ for some $N>0$, then ${\rm AVR}\ms(x_0, N')=0$ for any $N'\geq N$.
\end{lemma}

\begin{proof}
The uniqueness of $N$ is obvious, we will just show that ${\rm AVR}\ms(x_0, N)$ is independent of the choice of $x_0$.
Let $x_1\in X$ with $x_1\neq x_0$. Using the triangle inequality it is easy to see that
\begin{equation*}
 \frac {\mm\big (B_r(x_1)\big)} {r^N} \leq  \frac {\mm\big (B_{r+\d(x_0, x_1)}(x_0)\big)} {r^N}.
\end{equation*}
So 
\begin{eqnarray*}
\lmts{r}{+\infty}  \frac {\mm\big (B_r(x_1)\big)} {r^N} &\leq&  \lmts{r}{+\infty} \frac {\mm\big (B_{r+\d(x_0, x_1)}(x_0)\big)} {r^N}\\
&=&   \lmts{r}{+\infty} \frac {\mm\big (B_{r+\d(x_0, x_1)}(x_0)\big)}{\big(r+\d(x_0, x_1)\big)^N}  \frac {\big(r+\d(x_0, x_1)\big)^N} {r^N}\\
&=&  {\rm AVR}\ms(x_0, N).
\end{eqnarray*}

Similarly, we have
\begin{eqnarray*}
\lmti{r}{+\infty}  \frac {\mm\big (B_{r}(x_1)\big)} {r^N} &\geq&  \lmt{r}{+\infty} \frac {\mm\big (B_{r-\d(x_0, x_1)}(x_0)\big)} {r^N}\\
&=&   \lmt{r}{+\infty} \frac {\mm\big (B_{r-\d(x_0, x_1)}(x_0)\big)}{\big(r-\d(x_0, x_1)\big)^N}  \frac {\big(r-\d(x_0, x_1)\big)^N} {r^N}\\
&=&  {\rm AVR}\ms(x_0, N)
\end{eqnarray*}
and the conclusion follows.

\end{proof}
Thanks to the previous result, if there exists $x_0\in X$ and $N>0$ such that  ${\rm AVR}\ms(x_0,N)\in (0,+\infty)$ then we can write ${\rm AVR}\ms$ without any further specification. Moreover, if there is no risk of confusion we will just write ${\rm AVR}$ thus omitting also the dependence on the metric measure space. In case  ${\rm AVR}>0$,  we  say that $\ms$  has {\bf Euclidean-volume growth}.

\begin{lemma}\label{lemma1}
Let $\ms$ be a metric measure space admitting  finite asymptotic volume ratio, i.e. ${\rm AVR}\ms\in [0,+\infty)$, $p>1$ and let $(a_n)_{n\in \N}$ be a non-increasing  sequence  of strictly positive numbers converging to $0$. For each $n\in\mathbb{N}$ we define $\rho_n:(X\times X)\setminus \{(x,y)\in X\times X\ |\ x=y\}\to (0,\infty)$ by 
\begin{equation*}
\rho_n(x, y)= \frac{a_n}{\d(x, y)^{N+a_n p}},
\end{equation*}
where $N\in \N$ is as in Definition \ref{def:avr}.
Then $(\rho_n)_{n\in\N}$ satisfies Assumption \ref{assumption2}.  In particular,
\begin{equation*}
L=\frac N p {\rm AVR}\ms.
\end{equation*}
\end{lemma}
\begin{proof} Clearly $\rho_n(x,y)=\tilde\rho_n(\d(x,y))$ where $\tilde\rho_n :(0,\infty)\to (0,\infty)$ is defined by
\begin{equation}
\tilde \rho_n(r)=\frac{a_n}{r^{N+a_n p}}.
\end{equation}
Then $\tilde \rho_n(r) $ is decreasing in $r$ for every $n\in \N$.\\ 

Fix $r>1$. We define  $\varphi:(0,\infty)\to (0,\infty)$ by $\varphi(x):=\frac{x}{r^{N+x p}}$. It can be seen that
\[
\varphi'(x)=\frac{r^{N+x p}-x p r^{N+x p} \ln r }{r^{2N+2x p}}>0
\]
for sufficiently small $x$. So $\tilde \rho_n(r)$ is decreasing in $n$ (for large $n$).\\
\noindent

Given $n>m$, by assumption $a_m\geq a_n$, so  $$ \frac{\tilde \rho_n(r)}{\tilde \rho_m(r)}= \frac{a_n}{a_m} {r^{(a_m -a_n)p}}$$
which is increasing for $r>0$.\\
\noindent

For simplicity, we assume that the asymptotic volume ratio  ${\rm AVR}$ is positive, the case for ${\rm AVR}=0$ can be proved in a similar way.  
 For any $\epsilon>0$, there is $\delta_0>0$ such that 
$${\rm AVR}\big (1-\epsilon\big) r^N \leq \mm\big(B_r(x) \big) \leq  {\rm AVR}\big (1+\epsilon\big) r^N,~~~\forall r\geq \delta_0.$$
For simplicity, we can also write
\[
\mm\big(B_r(x) \big) ={\rm AVR} \big (1+O(\epsilon)\big) r^N.
\]

Let $\delta>\delta_0$. For any $n\in \N$ we define the function $\bar \rho_{n,\delta}: X\times X\to \R$ by 
\[
\bar \rho_{n,\delta}(x, y):=\left\{\begin{array}{ll}
 \rho_n(x, y)~~~\d(x,y)> \delta,\\
\\
\tilde \rho_n(\delta)~~~~~~0\leq \d(x,y)\leq \delta.
\end{array}
\right.
\] 
Then by the Cavalieri's formula (cf. \cite[Chapter 6]{AT-T}) we can write
\begin{eqnarray*}
&&\int_{B^c_\delta(x) }  \rho_n(x, y) \,\d \mm(y)+\tilde \rho_n(\delta) \mm\big(B_\delta(x) \big) \\
&=& \int_{X }  \bar \rho_{n,\delta}(x, y) \,\d \mm(y) \\
\text{by Cavalieri's formula}~&=& \int_0^{\tilde \rho_n(\delta)} \mm\big(B_{\tilde \rho_n^{-1}(r)}(x) \big) \, \d r\\
&=& \int_0^{\tilde \rho_n(\delta)} {\rm AVR} \big (1+O(\epsilon)\big) \big({\tilde \rho_n^{-1}(r)} \big)^N \, \d r\\
\text{let} ~t=\tilde \rho_n^{-1}(r) ~~&=&\big (1+O(\epsilon)\big)  {\rm AVR}  \int_{+\infty}^{\delta}  t^N  \tilde \rho'_n(t)\, \d t\\
\text{by integration by part} ~~&=&  \big (1+O(\epsilon)\big)  {\rm AVR} \left(\tilde \rho_n(\delta) \delta^N  -\int_{+\infty}^{\delta}  N t^{N-1}  \tilde \rho_n(t)\, \d t \right)\\
&=&  \big (1+O(\epsilon)\big)  {\rm AVR} \left(\tilde \rho_n(\delta) \delta^N  -\int_{+\infty}^{\delta}  N  \frac{a_n}{t^{1+a_n p}}\ \, \d t \right)\\    
&= & \big (1+O(\epsilon)\big)  {\rm AVR} \left(\tilde \rho_n(\delta) \delta^N  +\frac N p \delta^{-a_n p} \right)\\   
\end{eqnarray*}

Therefore,  for $\delta>\delta_0$,  
\[
\int_{B^c_\delta(x) }  \rho_n(x, y) \,\d \mm(y)=O(\epsilon)  {\rm AVR} \frac{a_n}{\delta^{a_n p}} +\big (1+O(\epsilon)\big){\rm AVR} \frac N p \delta^{-a_n p}.
\]
Thus
 \[
\lmt{\delta}{+\infty} \lmt{n}{\infty}\int_{B^c_\delta(x) }  \rho_n(x, y) \,\d \mm(y)=  \frac N p {\rm AVR},~~\forall x\in X
 \]
which is the thesis.
\end{proof}

\bigskip
\begin{remark}
The function $r \to \mm\big(B_r(x) \big)$ is increasing, 
so it is differentiable except for a set of Lebesgue measure $0$.  In many situations, for example when the space satisfies some generalized Bishop-Gromov comparison theorem (cf. \cite[Theorem 2.3]{S-O2} and Theorem \ref{rigidity1}), we 
have  a stronger asymptotic formula:
\[
\lmt{r}{+\infty}\frac{\frac{\d}{\d r}\mm\big(B_r(x) \big) }{N r^{N-1}} ={\rm AVR}.
\] 
Furthermore, when $(X, \d)$ is geodesic, it is not hard to prove that
\[
\frac{\d}{\d r}\mm\big(B_r(x) \big) =\mm^+\big(\partial B_r(x) \big), 
\] where $\mm^+$ is the canonical Minkowski content  defined by
 \begin{equation}\label{eq:minkow}
\mm^+(E):=\liminf_{\varepsilon\to 0}\frac{\mm(E^{\varepsilon})-\mm(E)}{\varepsilon}
\end{equation}
where $E$ is a Borel set and $E^{\varepsilon}:=\{x\in X \ :\ \exists~y\in E\ \mbox{such that} ~ \d(x,y)<\varepsilon\}$ is the $\varepsilon$-neighbourhood of $E$ with respect to the metric $\d$. 
\end{remark}
\bigskip

\bigskip
Combining Theorem \ref{th:MS2} and Lemma \ref{lemma1} we immediately  get the following theorem.
\begin{theorem}\label{th:MS}
Let $\ms$ be a metric measure space admits finite asymptotic volume ratio ${\rm AVR}\ms$. For any  $p>1$ and $s\in (0, 1)$, define $W^{s, p}\ms$ as the functions in $L^p\ms$ with finite (semi)-norm 
\[
\left( \int_{X} \int_{X} \frac{|u(x)-u(y)|^p}{\d(x, y)^{N+sp}} \,\d \mm( x)\,\d \mm(y)\right)^{\frac 1p}
\]
where $N$ is the number given in Definition \ref{def:avr}.

For any $u\in \cup_{0<\tau<1} W^{\tau, p}\ms$, there exists the limit
\begin{equation}\label{eq1.2:MS}
\mathop{\lim}_{s\downarrow 0} s \int_{X} \int_{X} \frac{|u(x)-u(y)|^p}{\d(x, y)^{N+sp}} \,\d \mm( x)\,\d \mm(y) =\frac {2N} p {\rm AVR} \|u\|^p_{L^p}.
\end{equation}

\end{theorem}

\bigskip
In the study of metric (Riemannian) geometry, there are some important spaces where the asymptotic volume ratio ${\rm AVR}=+\infty$. In these cases, we can consider instead the {volume entropy}, which is an important concept in both Riemannian  geometry (cf.  \cite{BessonEntropy})  and dynamical system (cf. \cite{ManningEntropy} ).
\begin{definition}
A metric measure space $\ms$ admits the \textbf{volume entropy} at $x_0\in X$, denoted by $h\ms(x_0)$,  provided
\[
h\ms(x_0):=\lmt{r}{+\infty} \frac {\ln \mm\big (B_r(x_0)\big)} {r}\in [0,\infty].
\]
\end{definition}
Proceeding as in Lemma \ref{lemma:avr}, it is not hard to prove the following result (cf.  \cite{VolumeEntropy2021} and the references therein).
\begin{lemma}
Let $\ms$ be a metric measure space admitting finite volume entropy  $h\ms(x_0)$ at $x_0\in X$. Then   $\ms$ admits finite volume entropy at any $x_1\in X$ and
\begin{equation}
h\ms(x_0)=h\ms(x_1).
\end{equation}
\end{lemma}
\begin{theorem}\label{th:MS1}
Let $\ms$ be a metric measure space with $h=h\ms\in (0,+\infty)$. Then for any  $p>1$ and $u\in L^p(X)$ such that $\int_{X} \int_{X} \frac{|u(x)-u(y)|^p}{e^{(h+\tau)\d(x, y)}} \,\d \mm( x)\,\d \mm(y)<+\infty$ for some $\tau>0$,   there exists the limit
\begin{equation}\label{eq1:MS1}
\mathop{\lim}_{s\downarrow 0} s\int_{X} \int_{X} \frac{|u(x)-u(y)|^p}{e^{(h+s)\d(x, y)}} \,\d \mm( x)\,\d \mm(y)={2h}  \|u\|^p_{L^p}.
\end{equation}

\end{theorem}
\begin{proof}
Similar to Lemma \ref{lemma1}, we can check that the mollifiers $\tilde \rho(s):=\frac {s} {e^{(h+s)r}}$  satisfy Assumption \ref{assumption2}, and
\[
\lmt{\delta}{\infty} \lmt{s}{0}\int_{B^c_\delta(x) } \frac{s}{e^{(h+s)\d(x, y)}} \,\d \mm(y)=  h,~~\forall x\in X
\] Then the assertion follows from Theorem \ref{th:MS2}.
\end{proof}
\bigskip

 \begin{example}
The range of applicability of our main theorem is pretty wide. We list below some relevant examples.
 
\begin{itemize} 
 
\item [1)]  {\bf Euclidean spaces}: it is known that ${\rm AVR}(\R^N, |\cdot|, \mathcal L^N)=\omega_N=\frac {|S^{N-1}|} N$ where $\omega_N=\frac{\pi^{\frac N 2}}{\Gamma(\frac N2+1)}$ denotes the volume of an $N$-dimensional unit ball and $|S^{N-1}|$ denotes its surface area. By Theorem \ref{th:MS} we  get Maz'ya-Shaposhnikova's original result \cite[Theorem 3]{MS}:
\[
\mathop{\lim}_{s\downarrow 0} s \int_{\R^N} \int_{\R^N} \frac{|u(x)-u(y)|^p}{|x- y|^{N+sp}} \,\d \mathcal{L}^N(x)\,\d \mathcal{L}^N(y) =\frac {2 |S^{N-1}|} p \|u\|^p_{L^p(\R^N)}.
\]

\item [2)]  {\bf Finite dimensional Banach spaces}:  let $(\R^N, \| \cdot \|, \mathcal L^N)$ be an $N$-dimensional Banach space. Denote its unit ball by $K$ (which is a convex body).   Applying Theorem \ref{th:MS},  we get Ludwig's result \cite[Theorem 2]{Ludwig14} for anisotropic fractional Sobolev norms:
\[
\mathop{\lim}_{s\downarrow 0} s\int_{\R^N} \int_{\R^N} \frac{|u(x)-u(y)|^p}{\|x- y\|^{N+sp}} \,\d \mathcal L^N(x)\,\d \mathcal L^N(y)=\frac {2N} p |K|\|u\|^p_{L^p(\R^N)}
\]
where $|K|$ denotes the volume of $K$.

\item [3)] {\bf  Riemannian manifolds}: 
Let $(M^N, g)$ be a complete Riemannian manifold of dimension $N$ with ${\rm Ric}\geq 0$. 
Let $\d$ be the distance determined by $g$, and let $\mm$ be the volume element determined by $g$.
Since  ${\rm Ric} \geq 0$ then, by the classical Bishop-Gromov Volume Comparison Theorem, the asymptotic volume ratio exists  so \eqref{eq1.2:MS} holds.
If $h(M, \d, \mm)\in (0, +\infty)$, then \eqref{eq1:MS1} holds.

\item[4)] {\bf Carnot groups}: Let $\mathbb{G}=(\mathbb{R}^n,\cdot)$ be a Carnot group of step $s$ endowed with the Carnot-Carath\'eodory distance $\d_{cc}$ and the Lebesgue measure $\mathcal{L}^n$ (we address the reader to \cite{BLU} for all the relevant definitions). It is well known that
$\mathcal{L}^n(B(x,r))=r^{N} \mathcal{L}^n(B(0,1))$ where $N\in \N$ is the so called homogeneous dimension of $\mathbb{G}$. It is then clear that ${\rm AVR}=\mathcal{L}^n(B(0,1))>0$ and \eqref{eq1.2:MS} holds.

\item[5)] {\bf {\rm MCP}$(0, N)$ spaces}:
Let $\ms$ be a metric measure space satisfying the so-called Measure Contraction Property {\rm MCP}$(0, N)$,  a property  introduced independently by Ohta \cite{Ohta-MCP} and Sturm \cite{S-O2} as synthetic curvature bound of metric measure spaces.
By Generalized Bishop--Gromov volume growth inequality (cf. \cite[Theorem 2.3]{S-O2}),   the limit
\[
{\rm AVR}=\lmt{r}{+\infty} \frac {\mm\big (B_r(x)\big)} {r^N}
\]
exists finite (it can be 0) and does not depend on the point $x \in X$.

 It was shown by Juillet \cite{Juillet09} that the $n$-dimensional Heisenberg group $\H^n$, which is the simplest example of a non-trivial sub-Riemannian manifold,  equipped with the Carnot-Carath\'eodory metric and the Lebesgue measure, satisfies ${\rm MCP}(0,N)$  for $N = 2n+3$.
Recently, interpolation inequalities \`a la Cordero-Erausquin--McCann--Schmuckenshl\"ager \cite{CMS01} have been obtained, under suitable modifications, by Balogh, Krist\'aly and Sipos  \cite{BKS18} for the Heisenberg group and by Barilari and Rizzi \cite{BarilariRizzi18} in the general ideal sub-Riemannian setting.  As a consequence, more  examples of spaces verifying ${\rm MCP}$ have been found, e.g. generalized H-type groups, the Grushin plane and Sasakian structures
(see \cite{BarilariRizzi18} for more details). 
\end{itemize}
\end{example}

\subsection{Rigidity results}\label{splitting}
In this subsection we apply Theorems \ref{th:MS} and Theorems \ref{th:MS1} to get  rigidity results for metric measure spaces satisfying the $\rcdkn$ condition.
It is well-known that $\rcdkn$ space are obtained by adding a Riemannian structure (called infinitesimally Hilbertian) to a metric measure space satisfying the ${\rm CD}(K,N)$ condition \`a la Lott-Sturm-Villani \cite{Lott-Villani09, S-O1, S-O2}. They have been introduced in \cite{AGS-M} by Ambrosio-Gigli-Savar\'e (when $N=\infty$) and by Gigli in \cite{G-S,G-O}  (treating ${\rm RCD}^*(K,N)$ and infinitesimally Hilbert spaces).

Important examples of spaces satisfying $\rcdkn$ and $\rcd$ conditions include: weighted Riemannian manifolds satisfying Bakry-\'Emery condition,  measured-Gromov Hausdorff limits of Riemannian manifolds with $\Ric  \geq  K $ (cf. \cite{Lott-Villani09, S-O1}),   Alexandrov spaces with curvature $\geq K $ (cf. \cite{ZhangZhu10}). We refer the reader to the ICM proceeding \cite{AmbrosioICM} by Ambrosio for an overview of this topic . 

\begin{theorem}\label{rigidity1}
Let $\ms$ be a ${\rm RCD}(0, N)$ metric measure space with $N\in (1, +\infty)$.  Then
\begin{equation}\label{eq1:rigidity1}
\mathop{\lim}_{s\downarrow 0} s\int_{X} \int_{X} \frac{|u(x)-u(y)|^p}{\d^{N+sp}(x, y)} \,\d \mm( x)\,\d \mm(y) \leq \frac {2N\omega_N} p  \|u\|^p_{L^p}.
\end{equation}
for any $u\in \cup_{\tau\in (0, 1)} W^{\tau, p}$,  where $\omega_N$ denotes the volume of an $N$-dimensional unit ball. 

If the equality in \eqref{eq1:rigidity1} is attained by a non-zero function $u$, then $N\in \N$ and $\ms$ is isometric to a metric 
cone over an ${\rm RCD}(N-2, N-1)$ space. 
\end{theorem}

\begin{proof}
Firstly, by Generalized Bishop--Gromov volume growth inequality (cf. \cite[Theorem 2.3]{S-O2}) we know 
\begin{equation}\label{eq2:rigidity1}
 {\rm AVR} \ms =\lmt{r}{\infty}\frac {\mm\big (B_r(x_0)\big)} {r^N}  \leq \frac {\mm\big (B_R(x_0)\big)} {R^N} \leq \omega_N,~~~\forall R>0.
\end{equation}
So  \eqref{eq1:rigidity1} follows from  Theorem \ref{th:MS}.

If the equality in \eqref{eq1:rigidity1} is attained by a non-zero function $u$, by Theorem \ref{th:MS} and \eqref{eq2:rigidity1} we can see that 
\[
\frac {\mm\big (B_R(x_0)\big)} {R^N} =\omega_N,~~~\forall R>0.
\]
By \cite[Theorem 1.1]{DPG-F} we know $\ms$ is isometric to a metric 
cone over an ${\rm RCD}(N-2, N-1)$ space. Furthermore, by \cite{BrueSemolaConstant} and \cite{DPG-Non} we know $N$ is an integer.
\end{proof}

\bigskip

\begin{theorem}\label{rigidity2}
Let $\ms$ be a non-compact ${\rm RCD}(-(N-1), N)$ metric measure space with $N\in (1, +\infty)$.  Then
\begin{equation}\label{eq1:rigidity2}
\mathop{\lim}_{s\downarrow 0} s\int_{X} \int_{X} \frac{|u(x)-u(y)|^p}{e^{(h+s)\d(x, y)}} \,\d \mm( x)\,\d \mm(y) \leq  {2(N-1)}  \|u\|^p_{L^p}
\end{equation}
for any $u\in \Lip_b(X, \d)$ with bounded support.

If the equality in \eqref{eq1:rigidity2} is attained by a non-zero function $u$, then  $\ms$ is isometric to  a warped product space $\R \times_{e^t} X'$, where $X'$ is an ${\rm RCD}(0, N)$ space.
\end{theorem}
\begin{proof}
By \cite[Corollary 3.2]{VolumeEntropy2021} we know $h\ms \leq N-1$. So  \eqref{eq1:rigidity2} follows from  Theorem \ref{th:MS1}.

If the equality in \eqref{eq1:rigidity2} is attained by a non-zero function $u$, by Theorem  \ref{th:MS1} and the inequality $h\ms \leq N-1$ we have  $h\ms = N-1$.
Then by \cite[Theorem 1.2]{VolumeEntropy2021}, $\ms$ is isometric to  a warped product space $\R \times_{e^t} X'$, where $X'$ is an ${\rm RCD}(0, N)$ space.
\end{proof}
\def\cprime{$'$}

\end{document}